\documentclass[a4paper,12pt,reqno]{amsart}
\usepackage{amsmath,amssymb,amsthm}
\usepackage{leftidx}
\usepackage{mathptmx}
\usepackage{mathtools}
\usepackage{mathrsfs}
\theoremstyle{plain}
\newtheorem{thm}{Theorem}[section]
\newtheorem{corollary}[thm]{Corollary}
\newtheorem{lemma}[thm]{Lemma}

\theoremstyle{definition}
\newtheorem{definition}{Definition}

\theoremstyle{remark}
\newtheorem{rem}{Remark}[section]

\newtheorem{example}{Example}

\numberwithin{equation}{section}
\begin{document}
\begin{center}{\bf{Arithmetic summable sequence space over non-Newtonian field}}

\vspace{.5cm}
Taja Yaying $^{1}$, Bipan Hazarika $^2$\\
$^{1}$Department of Mathematics, Dera Natung Govt. College, Itanagar-791111, Arunachal Pradesh, India\\
$^{2}$Department of Mathematics, Guwahati University, Guwahati, Assam, India\\
Email:tajayaying20@gmail.com;bhrgu@yahoo.co.in
\end{center}
\title{}
\author{}
\thanks{}
\date{}
\begin{abstract} Recently Ruckle \cite{RuckleArithmeticalSummability} introduced the theory of arithmetical summability suggested by the sum $ \sum_{k|m}f(k) $ as $ k $ ranges over the divisors of $m$ including $ 1 $ and $ m .$ Following Ruckle \cite{RuckleArithmeticalSummability} we construct the sequence space $ AS(G) $ and $ AC(G) $ of arithmetic summable and arithmetic convergent sequences in the sense of geometric calculus and derive interesting results in the geometric field.  
{\footnotesize {AMS subject classification \textrm{(2010)}:}} Primary: 26A06; Secondary: 46A45, 46A11.
\end{abstract}
\maketitle
\section{Introduction}
Grossman and Katz \cite{GrossmanandKatzNonNewtonianCalculus} first introduced non-Newtonian calculus as an alternative to the classical calculus. Non-Newtonian calculus has many branches such as bigeometric calculus, geometric calculus, harmonic calculus, etc. Any type of calculi other than the classical calculus are together termed as non-Newtonian calculus. Every property in a non-Newtonian calculus has an analogue in the classical calculus that provides an alternative structure to the problems that can be investigated through classical calculus. In some cases, for example growth related problems, the use of non-Newtonian calculus is advocated instead of the traditional one. \\
Bashirov et. Al. \cite{BashirovKurpinarOzyapici} introduced the notions of multiplicative derivatives and multiplicative integrals and give their applications. Many authors \cite{CakmakBasar, CakmakBasar2, TekinBasar, KadakKotheToeplitzduals, KadakandEfe} have constructed different types of sequence spaces over non-Newtonian field. Talo and Ba\c{s}ar \cite{TaloBasar} have studied the certain sets of sequences of fuzzy numbers and introduced the classical sets $\ell_{\infty}(F)$, $c(F)$, $c_0(F)$ and $\ell_p(F)$ consisting of the bounded, convergent, null and absolutely $p$-summable sequences of fuzzy numbers. They defined the $\alpha$, $\beta$ and $\gamma$ duals of a set of sequences of fuzzy numbers, and gave the duals of the classical sets of sequences of fuzzy numbers together with the characterization of the classes of infi…nite matrices of fuzzy numbers transforming one of the classical set into another one. \c{C}akmak and Ba\c{s}ar \cite{CakmakBasar} constructed the field $\mathbb{R}(N)$ of non-Newtonian real numbers and show that $\mathbb{R}(N)$ is a complete field. They define the sets $w(N),$ $\ell_{\infty}(N)$, $c(N)$, $c_0(N)$, and $\ell_p(N)$ of all, bounded, convergent, null and $p$-summable sequence spaces respectively in the sense of non-Newtonian calculus and show that each of the sets
forms a vector space on the field $\mathbb{R}(N)$ and a complete metric space. T\"{u}rkmen and Ba\c{s}ar \cite{TurkmenandBasar} define the sequence spaces $w(G)$, $\ell_{\infty}(G)$, $c(G)$, $c_0(G)$, and $\ell_p(G)$ of all, bounded, convergent, null and $p$-summable sequence spaces respectively over geometric complex field $\mathbb{C}(G)$ and show that each of the sequence spaces forms a vector space on the field $\mathbb{C}(G)$ and a complete metric space. Tekin and Ba\c{s}ar \cite{TekinBasar} extended these results to construct the sequence spaces $w^*$, $\ell_{\infty}^*$, $c^*$, $c_0^*$ and $\ell_p^*$ over the non-Newtonian complex field $\mathbb{C}^*.$ More recently, Kadak \cite{kadakCesaroSummableSeqnSpace} introduced the Ces\`{a}ro *-summable sequence spaces $w_0^p(\alpha,\beta)$, $w^p(\alpha, \beta)$ and $w_{\infty}^p(\alpha, \beta)$ that are strongly *-summable to zero, *-summable and *-bounded respectively over non-Newtonian complex field. One may refer to \cite{GrossmanBigeometricCalculus, GrossmanandKatzNonNewtonianCalculus, BashirovKurpinarOzyapici, BashirovRiza, CakmakBasar, KadakandEfe, UzerMultiCalulus, kadakCesaroSummableSeqnSpace, TaloBasar} for greater inside into non-Newtonian calculus.\\
We use the notation $\sum_{k|n} f(k)$ to denote the finite sum of all the numbers $f(k)$ as $k$ ranges over the integers that divide $n$ including $1$ and $n.$ For two integers $m$ and $n$ the greatest common divisor of $m$ and $n$ denoted as $\left\langle m,n\right\rangle$ is the largest number that divides both $m$ and $n.$\\
Ruckle \cite{RuckleArithmeticalSummability} introduced the notions of arithmetical summability and arithmetic convergence as\\
\begin{enumerate}
\item[(1)] A sequence $f$ defined on $\mathbb{N}$ is called arithmetically summable if for each $\varepsilon>0$ there is an integer $n$ such that for every integer $m$ we have $|\sum_{k|m}f(k)-\sum_{k|\left\langle m,n \right\rangle}f(k)|<\varepsilon.$\\
\item[(2)] A sequence $g$ is called arithmetically convergent if for each $\varepsilon>0$ there is an integer $n$ such that for every integer $m$ we have $|g(m)-g(\left\langle m,n\right\rangle)| < \varepsilon.$
\end{enumerate} 
Yaying and Hazarika \cite{YayingHazarikaArContinuity} further extended the theory of arithmetical summability and introduced the concepts of arithmetic continuity and arithmetic compactness in $\mathbb{R}.$ One may refer to \cite{RuckleArithmeticalSummability, YayingHazarikaArContinuity, YayingHazarikaArModulusFn, YayingHazarikaArSummMultiplier, cakalliVariationArContinuity} for greater inside to arithmetical summability.\\
The main purpose of this article is to investigate the theory of arithmetical summability using the sum $\sideset{_G}{}\sum_{k|m}f(k)$ i.e. the geometric sum of the geometric numbers $f(k)$ as $k$ ranges over the divisors of $m,$ accompanied by the theory of arithmetic convergence over the geometric real field and give some interesting results.
\section{Preliminaries, Background, and Notations}
An arithmetic is any system that satisfies the whole of the ordered field axioms whose domain is a subset $\mathbb{R}(\alpha)$ of $\mathbb{R},$ the set of real numbers. There are infinitely many types of arithmetic, all of which are isomorphic, that is, structurally equivalent.\\
A generator, denoted by, $\alpha,$ is a one-to-one function whose domain is $\mathbb{R}$ and range is a subset of $\mathbb{R}.$ Each generator generates exactly
one arithmetic and, conversely, each arithmetic is generated by exactly one generator.\\
By $\alpha$-arithmetic we mean the arithmetic whose domain is $\mathbb{R}(\alpha).$ Consider any generator $\alpha$ and $x, y\in \mathbb{R}(\alpha),$\\
\begin{eqnarray*}
\alpha-addition: x\dot{+}y=\alpha\left\lbrace \alpha^{-1}(x)+\alpha^{-1}(y)\right\rbrace \\
\alpha-subtraction: x\dot{-}y=\alpha\left\lbrace \alpha^{-1}(x)-\alpha^{-1}(y)\right\rbrace \\
\alpha-multiplication: x\dot{\times}y=\alpha\left\lbrace \alpha^{-1}(x)\times\alpha^{-1}(y)\right\rbrace \\
\alpha-division: x\dot{/}y=\alpha\left\lbrace \alpha^{-1}(x)\div \alpha^{-1}(y)\right\rbrace \\
\end{eqnarray*}
If we take $\alpha$ as an Identity function, then $\alpha$-arithmetic reduces to the classical arithmetic. In our case, we shall consider $\alpha$ to be an exponential function i.e. $\alpha:\mathbb{R}\rightarrow\mathbb{R^+}$ given by $\alpha(x)=e^x,$ which gives $\alpha^{-1}(x)=\ln x.$ Then $\alpha$-arithmetic reduces to what we call geometric arithmetic as given below:
\begin{eqnarray*}
\alpha-addition : x\oplus y=\alpha\left\lbrace \alpha^{-1}(x)+\alpha^{-1}(y)\right\rbrace =e^{\left\lbrace \ln x+\ln y\right\rbrace }=x. y~:Geometric ~addition\\
\alpha-subtraction : x\ominus y=\alpha\left\lbrace \alpha^{-1}(x)-\alpha^{-1}(y)\right\rbrace =e^{\left\lbrace \ln x-\ln y\right\rbrace }=x\div y~: Geometric~subtraction\\
\alpha-multiplication: x\odot y=\alpha\left\lbrace \alpha^{-1}(x)\times\alpha^{-1}(y)\right\rbrace=e^{\left\lbrace \ln x\times \ln y\right\rbrace}=x^{\ln y}~: Geometric~ multiplication \\
\alpha-division: x\oslash y=\alpha\left\lbrace \alpha^{-1}(x)\div \alpha^{-1}(y)\right\rbrace=e^{\left\lbrace \ln x\div \ln y\right\rbrace}=x^{\frac{1}{\ln y}}~:Geometric ~division  \\
\end{eqnarray*}
The set $\mathbb{Z}(G)$ of geometric integers and the set $\mathbb{R}(G)$ of geometric real numbers are given by
\begin{eqnarray*}
\mathbb{Z}(G)=\left\lbrace e^x: x\in \mathbb{Z}\right\rbrace \\
\mathbb{R}(G)=\left\lbrace e^x: x\in \mathbb{R}\right\rbrace= \mathbb{R}^+\setminus \left\lbrace 0\right\rbrace  \\
\end{eqnarray*}

\begin{rem}
$\left( \mathbb{R}(G),\oplus,\odot\right)$  is a complete field with geometric unity $e$ and geometric zero $e^0=1.$
\end{rem}
\begin{definition}\cite{TurkmenandBasar}
A sequence $(x_n)$ in a metric space $(X, d_G)$ is said to be convergent in the sense of geometric calculus if there is an $x\in X$ such that $\lim_{n\rightarrow\infty}d_G(x_n, x) = 1$ then $x$ is called the geometric limit of $(x_n)$ and we write $\leftidx{^G}\lim_{n\rightarrow \infty}x_n = x$ or $x_n\xrightarrow[]{G} x;$ as $n\rightarrow \infty.$
\end{definition}
\begin{definition}\cite{TurkmenandBasar}
A sequence $(x_n)$ in a geometric metric space $(X, d_G)$ is said to be geometric Cauchy if for every $\varepsilon >1$ in $\mathbb{R}^+(G)$, there is an $N = N(\varepsilon)$ such that $\left|x_m\ominus x_n \right|<\varepsilon $ for all $m,n>N.$ The space $X$ is said to be geometric complete if every geometric Cauchy sequence converges in $X.$
\end{definition}
Now we are ready to define arithmetic summable and arithmetic convergence sequence space in the sense of geometric calculus.
\section{Arithmetic summable and arithmetic convergent sequence space over geometric real field}
\begin{definition}
A geometric sequence $f$ is said to be arithmetic summable in the sense of geometric calculus if there exist an integer $n$ such that for all $m\in \mathbb{N}$ we have 
\[
\leftidx{_G}{\left|\sideset{_G}{}\sum_{k|m}f(k)\ominus \leftidx{_G}{\sum}_{k|{\left\langle m,n\right\rangle}} f(k)\right|}=1
\]
\end{definition}
\begin{definition}
A geometric sequence $g$ is said to be arithmetic convergent in the sense of geometric calculus if there exist an integer $n$ such that for all $m\in \mathbb{N}$ we have
\begin{equation*}
\leftidx{_G}{\left|g(m)\ominus g(\left\langle m,n\right\rangle)\right|}=1
\end{equation*}
\end{definition}
We shall denote the set of all arithmetic summable sequences and arithmetic convergent sequences in the sense of geometric calculus by $AS(G)$ and $AC(G)$ respectively. \\
Consider the sequence $\left\lbrace e,e,e^0,e^0,e^0,\ldots\right\rbrace $ in $\mathbb{R}(G).$ Clearly the sequence is arithmetic summable in $\mathbb{R}(G).$ In fact, the sequences $\leftidx{_G}{\varphi}$ i.e. the space of the all sequences $f(k)$ in $\mathbb{R}(G)$ such that there is $\leftidx{_G}{N}$ for which $f(k)=e^0$ for all $k>\leftidx{_G}{N}$ is arithmetic summable i.e. $\leftidx{_G}{\varphi}\subset AS(G).$\\
\begin{thm}
The sequence spaces $AC(G)$ and $AS(G)$ are linear in the sense of geometric calculus.
\end{thm}
\begin{proof}
Let $f(i)$ and $g(i)$ be any two sequences in $\mathbb{R}(G).$ Thus, for an integer $n,$ we have,\\
\begin{equation*}
\leftidx{_G}{\left|f(i)\ominus f(\left\langle i,n\right\rangle)\right|}=1~\text{and}~\leftidx{_G}{\left|g(i)\ominus g(\left\langle i,n\right\rangle)\right|}=1,\quad \text{for all $i.$}
\end{equation*}
Thus for integer $n$ and for scalars $\alpha, \beta\in \mathbb{R}(G),$
\begin{align*}
\leftidx{_G}{\left|\left( \alpha\odot f(i)\oplus\beta\odot g(i)\right) \ominus \left( \alpha \odot f(\left\langle i,n\right\rangle)\oplus \beta \odot g(\left\langle i,n\right\rangle)\right) \right|}
&= \leftidx{_G}{\left|\alpha\odot \left(f(i)\ominus f(\left\langle i,n\right\rangle)\right)\oplus \beta \odot \left(g(i)\ominus g(\left\langle i,n\right\rangle)\right) \right|}\\
&\leq \leftidx{_G}{\left|\alpha\right| }\odot \leftidx{_G}{\left|f(i)\ominus f(\left\langle i,n\right\rangle)\right|}\oplus \leftidx{_G}{\left|\beta \right| } \odot\\
& \leftidx{_G}{\left|g(i)\ominus g(\left\langle i,n\right\rangle)\right|}\\
&\leq \leftidx{_G}{\left|f(i)\ominus f(\left\langle i,n\right\rangle)\right|}\oplus \leftidx{_G}{\left|g(i)\ominus g(\left\langle i,n\right\rangle)\right|}\\
&= e^0\oplus e^0=e^0
\end{align*}
This shows that the sequence space $AC(G)$ is a linear space.\\
Similarly we can show that $AS(G)$ is also linear.
\end{proof}
The following theorem helps in converting any sequence in $\mathbb{R}(G)$ into a sequence in $AC(G).$
\begin{thm}
If $f$ is any sequence in $\mathbb{R}(G)$ and $n$ any integer then the sequence $\mathcal{Q}_nf=g$ defined by $g(i)=f(\left\langle n,i \right\rangle)$ is in $AC(G).$   
\end{thm}
\begin{proof}
We have, for an integer $n$ and for all $m\in\mathbb{N}$
\begin{eqnarray*}
\leftidx{_G}{\left| g(m)\ominus g(\left\langle m,n \right\rangle)\right| }
&=&\leftidx{_G}{\left| f(\left\langle m,n\right\rangle)\ominus f(\left\langle n,\left\langle m,n\right\rangle \right\rangle)\right| }\\
&=&\leftidx{_G}{\left| f(\left\langle m,n\right\rangle)\ominus f(\left\langle m,n\right\rangle)\right| }\\
&=& e^0
\end{eqnarray*}
This shows that the sequence $g\in AC(G).$
\end{proof}  

\begin{example}\label{example1}
We consider a sequence $f=\{e,e^2,e^3,e^4,\ldots\}$ in $\mathbb{R}(G)$ and taking $n=1,$\\
\begin{eqnarray*}
\mathcal{Q}_1f(1)=g(1)=f(\left\langle 1,1 \right\rangle)=f(1)=e;\\
\mathcal{Q}_1f(2)=g(2)=f(\left\langle 2,1 \right\rangle)=f(1)=e;\\
\mathcal{Q}_1f(3)=g(3)=f(\left\langle 3,1 \right\rangle)=f(1)=e;\\
\text{and so on.}
\end{eqnarray*}
Thus, $g=\{e,e,e,\ldots\}$ which is clearly arithmetic convergent.
\end{example}
The next result helps in transforming the arithmetic convergent sequences in $\mathbb{R}(G)$ to ordinary convergence in $\mathbb{R}(G).$
\begin{thm}
If $f\in AC(G)$  and $n_k$ is a sequence of integers such that $n_k|n_{k+1}$ for all $k$ then the sequence $g(k)=f(n_k)$ converges in $\mathbb{R}(G).$
\end{thm}
\begin{proof}
Consider an integer $n$ be such that $\leftidx{_G}{\left| f(i)\ominus f(\left\langle i,n\right\rangle)\right|} =e^0$ for all $i.$\\
We choose $K$ such that $\left\langle n,n_k\right\rangle=\left\langle n,n_{k+1}\right\rangle$ for all $k\geq K.$\\
Thus for all $i,j\geq K,$
\begin{eqnarray*}
\leftidx{_G}{\left| g(i)\ominus g(j)\right| }&=&\leftidx{_G}{\left|f(n_i)\ominus f(n_j)\right| }\\
&=&\leftidx{_G}{\left|f(n_i)\ominus f(\left\langle m,n_i\right\rangle)\oplus f(\left\langle m,n_i\right\rangle)\ominus f(n_j)\right| }\\
&\leq& \leftidx{_G}{\left|f(n_i)\ominus f(\left\langle m,n_i\right\rangle)\right|} \oplus \leftidx{_G}{\left| f(\left\langle m,n_i\right\rangle)\ominus f(n_j)\right|}\\
&=& \leftidx{_G}{\left|f(n_i)\ominus f(\left\langle m,n_i\right\rangle)\right|} \oplus \leftidx{_G}{\left|f(n_j)\ominus f(\left\langle m,n_j\right\rangle)\right|}\\
&=& e^0\oplus e^0=e^0
\end{eqnarray*}
Thus the sequence $g$ satisfies the Cauchy's criterion in $\mathbb{R}(G)$ and so converges.
\end{proof}
\begin{example}
We take a sequence $f=\{e,e^2,e,e^2,\ldots\}.$ It is easy to verify that the sequence $f\in AC(G)$(refer example \ref{example1} and take $n=2$). Consider $n_k=\{2,4,8,16\ldots\}.$\\
Now,
\begin{align*}
g(1)=f(n_1)=f(2)=e^2;\\
g(2)=f(n_2)=f(4)=e^2;\\
g(3)=f(n_3)=f(8)=e^2;\\
\text{and so on.}
\end{align*}
Thus $g=\{e^2,e^2,e^2,\ldots\}$ which clearly converges in $\mathbb{R}(G).$
\end{example}
Before we move to the next results of the paper, we define analogue of M\"{o}bius function in the sense of geometric calculus and give some corresponding well known results of classical calculus.
\begin{definition}[Analogous of M\"{o}bius function in $\mathbb{R}(G)$]
M\"{o}bius function in the sense of geometric calculus is a mapping $\leftidx{_G}{\mu}:\mathbb{N}\rightarrow \{e,e^{(-1)^k},e^0\}$ defined by
\begin{eqnarray*}
\leftidx{_G}{\mu}(n)=
\begin{cases}
e,\qquad &\text{when}~n=1;\\
e^{(-1)^k},\qquad &\text{when}~n=p_1p_2\ldots p_k;\\
e^0, \qquad &\text{when $n$ has a square prime factor;}\\
\end{cases}
\end{eqnarray*}
where $p_i$'s are distinct prime numbers.
\end{definition}
\begin{lemma}\label{resultmu}
\begin{eqnarray*}
\sideset{_G}{}\sum_{k|n}\leftidx{_G}{\mu}(k)=
\begin{cases}
e, \qquad & n=1;\\
e^0, \qquad & n>1;\\
\end{cases}
\end{eqnarray*}
\end{lemma}
\begin{proof}
The result is obvious when $n=1.$\\
Let $n>1.$ Without loss of generality, we consider $n=p_1p_2\ldots p_r.$\\
\begin{eqnarray*}
\sideset{_G}{}\sum_{k|n}\leftidx{_G}{\mu}(k)
&=& \leftidx{_G}{\mu}(1)\oplus \sideset{_G}{}\sum_{i=1}^{r}\leftidx{_G}{\mu}(p_i)\oplus \sideset{_G}{}\sum_{i<j}^{r}\leftidx{_G}{\mu}(p_ip_j)\oplus\ldots \oplus\leftidx{_G}{\mu}(p_1p_2\ldots p_r)\oplus e^0\oplus\ldots\oplus e^0\\
&=& e\oplus e^{r(-1)}\oplus e^{^rC_2(-1)^2}\oplus \ldots \oplus e^{(-1)^r}\\
&=& e\cdot e^{r(-1)}\cdot e^{^rC_2(-1)^2}\ldots e^{(-1)^r}\\
&=& e^{1+^rC_1(-1)+^rC_2(-1)^2+\ldots +(-1)^r}\\
&=& e^{(1-1)^r}\\
&=& e^0\\
\end{eqnarray*}
Hence, the result.
\end{proof}
\begin{definition}[Dirichlet's product in $\mathbb{R}(G)$]
Let $f,g:\mathbb{N}\rightarrow \mathbb{R}(G).$ Then the Dirichlet's product of $f$ and $g$ is given by
\begin{equation*}
(f\ast g)(n)= \sideset{_G}{}\sum_{k|n}f\left( \frac{n}{k}\right) \odot g(k)
\end{equation*}
\end{definition}
Here, we note that the product $(\ast)$ is commutative and associative.\\
Next we construct an identity function with respect to Dirichlet's product in the sense of geometric calculus:
\begin{equation*}
\dot{e}(n)=
\begin{cases}
e, &\quad n=1;\\
e^0, &\quad n>1;
\end{cases}
\end{equation*}
It is easy to show that, $f\ast \dot{e}=f=\dot{e}\ast f.$ $\dot{e}$ will be called Dirichlet's identity in the sense of geometric calculus.\\
Consider the function $\dot{I}(n)=e$ for all $n$. It is interesting to see that 
\begin{equation*}
(\dot{I}\ast f)(n)=\sideset{_G}{}\sum_{k|n}f(k)=(f\ast \dot{I})(n)
\end{equation*}
A function $f$ is said to be the Dirichlet's inverse of $g$ if $f\ast g=\dot{e}=g\ast f.$\\
The following result is the immediate consequence of the Lemma \ref{resultmu}.
\begin{corollary}
The M\"{o}bius function $\leftidx{_G}{\mu}$ is Dirichlet's inverse of $\dot{I}$ in the sense of geometric calculus.
\end{corollary}
\begin{thm}[Analogous of M\"{o}bius Inversion theorem]\label{mit}
Let $f,g:\mathbb{N}\rightarrow \mathbb{R}(G).$ If $g(n)=\sideset{_G}{}\sum_{k|n}f(k)$ then
\begin{equation*}
f(n)=\sideset{_G}{}\sum_{k|n}\leftidx{_G}{\mu}\left( \frac{n}{k}\right) \odot g(k).
\end{equation*} 
\end{thm}
\begin{proof}
It is equivalent to show that $f=\leftidx{_G}{\mu}\ast g,$\\
Now,
\begin{equation*}
\leftidx{_G}{\mu}\ast g= \leftidx{_G}{\mu}\ast (\dot{I}\ast f)= (\leftidx{_G}{\mu}\ast \dot{I})\ast f = \dot{e}\ast f=f.
\end{equation*}
\end{proof}
Consider the matrix $W=[w_{i,j}]$ for which $w_{i,j}=e$ if $i$ is a divisor of j and $w_{i,j}=e^0$ otherwise. Thus the statement $g(n)=\sideset{_G}{}\sum_{k|n}f(k)$ is equivalent to $g=Wf$ where the juxtaposition denotes the matrix multiplication in the sense of geometric calculus. Using Theorem \ref{mit}, if $g=Wf$ then $f=Mg$ where $M$ is the matrix with element $\leftidx{_G}{\mu}\left(\frac{j}{i} \right)$ in the row $i$ and column $j$ if $i|j$ and $e^0$ otherwise. Thus\\
\[
W=
\begin{bmatrix}
e & e & e & e & e & e & e & e &\ldots\\
e^0 & e & e^0 & e & e^0 & e & e^0 & e &\ldots\\
e^0 & e^0 & e & e^0 & e^0 & e & e^0 & e^0 &\dots\\
e^0 & e^0 & e^0 & e & e^0 & e^0 & e^0 & e &\dots\\
\vdots & \vdots & \vdots & \vdots & \vdots & \vdots & \vdots & \vdots
\end{bmatrix}
\]
It is interesting to see that the rows of the matrix $W$ are sequences in $AC(G).$
\begin{thm}
If the sequence $f\in AS(G)$ then the sequence $g$ defined by $g(m)=\sideset{_G}{}\sum_{k|m}f(k)$ is in $AC(G).$
\end{thm}
\begin{proof}
Let $f\in AS(G).$ Then by definition, for an integer $n$\\
\begin{align*}
\leftidx{_G}{\left|\sideset{_G}{}\sum_{k|m}f(k)\ominus \sideset{_G}{}\sum_{k|{\left\langle m,n\right\rangle}}f(k)\right|}=1\\
\Rightarrow \leftidx{_G}{\left| g(m)\ominus g(\left\langle m,n \right\rangle)\right| }=1\\
\end{align*}
Thus $g$ is a sequence in $AC(G).$
\end{proof}
\begin{corollary}
If $f$ is a sequence in $AS(G)$ then $Wf$ is a sequence in $AC(G)$ and if $g$ is a sequence in $AC(G)$ then $Mg$ is a sequence in $AS(G).$
\end{corollary}
Now we focus on some basic properties of the sequence space $AS(G).$
\begin{thm}
If $f$ is any sequence in $\mathbb{R}(G)$ and $n$ any integer then the sequence $\mathcal{R}_nf=g$ defined by $g(i)=f(i)$ when $i|n$ and $e^0$ when $i\nmid n,$ is in $\leftidx{_G}{\varphi}.$ Thus the sequence $\mathcal{R}_n\in AS(G).$
\end{thm}
\begin{proof}
The proof is a routine and hence omitted.
\end{proof}
\begin{example}
We consider a sequence $f=\left\lbrace e, e^2, e^3, e^4,\ldots \right\rbrace$ in $\mathbb{R(G)}.$\\
Then, by above definition;\\
\begin{align*}
& R_1=\left\lbrace e, e^0,e^0,e^0,\ldots \right\rbrace \\
& R_2=\left\lbrace e,e^2, e^0, e^0,e^0,\ldots\right\rbrace \\
& R_3= \left\lbrace  e, e^0, e^3, e^0,e^0,e^0,\ldots\right\rbrace \\
\vdots\\
& \text{and so on.}
\end{align*}
which are sequences in $\leftidx{_G}{\varphi}$ and thus sequence in $AS(G).$
\end{example}
\begin{thm}
If $f\in AS(G)$ and $n_k$ is any sequence of integers such that $n_k|n_{k+1}$ then the sequence $g(k)=\sideset{_G}{}\sum_{d|n_k}f(d)$ converges in $\mathbb{R}(G).$
\end{thm}
\begin{proof}
Consider an integer $m$ such that for all $i\in \mathbb{N},$ 
\[\leftidx{_G}{\left| \sideset{_G}{}\sum_{d|i}f(d)\ominus \sideset{_G}{}\sum_{d|\left\langle i,m\right\rangle}f(d)\right|}=1\]
Choose $K$ such that $\left\langle m,n_k\right\rangle=\left\langle m,n_{k+1}\right\rangle  $ for all $k\geq K.$\\
Thus for all $i,j\geq K,$ we have,
\begin{eqnarray*}
\leftidx{_G}{\left| g(i)\ominus g(j)\right|} &=& \leftidx{_G}{\left| \sideset{_G}{}\sum_{d|n_i}f(d)\ominus \sideset{_G}{}\sum_{d|n_j}f(d)\right|}\\
&=& \leftidx{_G}{\left| \sideset{_G}{}\sum_{d|n_i}f(d)\ominus \sideset{_G}{}\sum_{d|\left\langle n_i,m\right\rangle}f(d)\oplus \sideset{_G}{}\sum_{d|\left\langle n_i,m\right\rangle}f(d)\ominus \sideset{_G}{}\sum_{d|n_j}f(d)\right|}\\
&\leq & \leftidx{_G}{\left| \sideset{_G}{}\sum_{d|n_i}f(d)\ominus \sideset{_G}{}\sum_{d|\left\langle n_i,m\right\rangle}f(d)\right|}\oplus \leftidx{_G}{\left| \sideset{_G}{}\sum_{d|\left\langle n_i,m\right\rangle}f(d)\ominus \sideset{_G}{}\sum_{d|n_j}f(d)\right|}\\
&=& e^0\oplus e^0=e^0.
\end{eqnarray*}
Thus $g$ satisfies the Cauchy's criterion and so converges in $\mathbb{R}(G).$
\end{proof}
\section{Conclusion}
In this article, we have investigated the characterizations of the sequence spaces $AC(G)$ and $AS(G)$ demonstrated with suitable examples in the sense of geometric calculus corresponding to the results obtained by Ruckle \cite{RuckleArithmeticalSummability} in the classical calculus. Analogous of the M\"{o}bius matrix and M\"{o}bius inversion formulae in the sense of geometric calculus are also given. We expect that our given notions, definitions, results and the realted investigations might be useful for others in modelling various problems in the field of economics, engineering, numerical analysis, mathematical physics, fuzzy theory, etc. As a future reference, one may obtain the duals of the spaces $AC(G)$ and $AS(G)$ and their relationship with the classical sequence spaces $\ell_{1}(G), \ell_{\infty}(G),$ etc. Further one may obtain similar results by employing other types of calculus.   

\end{document}